\documentclass[oneside,english]{amsart}
\usepackage[T1]{fontenc}
\usepackage[latin9]{inputenc}
\usepackage{amsthm}
\usepackage{amstext}
\usepackage{amssymb}
\usepackage{esint}
\usepackage{amscd}
\makeatletter
\numberwithin{equation}{section}
\numberwithin{figure}{section}
\theoremstyle{plain}
\newtheorem{thm}{\protect\theoremname}[section]

\theoremstyle{plain}
\theoremstyle{definition}

\theoremstyle{plain}

\theoremstyle{plain}

\theoremstyle{plain}
\makeatother

\usepackage{babel}
\providecommand{\definitionname}{Definition}
\providecommand{\lemmaname}{Lemma}
\providecommand{\theoremname}{Theorem}
\providecommand{\corollaryname}{Corollary}
\providecommand{\remarkname}{Remark}
\providecommand{\propositionname}{Proposition}

\DeclareMathOperator{\loc}{loc}

\DeclareMathOperator{\cp}{cap}
\DeclareMathOperator{\ACL}{ACL}
\DeclareMathOperator{\BMO}{BMO}

\usepackage{color}

\begin{document}

\title[Composition operators and weighted moduli inequalities]{Composition operators on Sobolev spaces and weighted moduli inequalities}

\author{V.~Gol'dshtein, E.~Sevost'yanov, A.~Ukhlov}
\begin{abstract}
In this paper we study connections between composition operators on Sobolev spaces and  mappings defined by $p$-moduli inequalities ($p$-capacity inequalities). We prove that weighted moduli inequalities lead to composition operators on corresponding Sobolev spaces and inverse, composition operators on Sobolev spaces imply weighted moduli inequalities.
\end{abstract}
\maketitle

\section{Introduction }

In this paper we study connections between composition operators on Sobolev spaces and  mappings 
defined by weighted $p$-moduli inequalities of curves families or corresponding weighted $p$-capacity inequalities. In the case $p=n$ the mappings defined by conformal moduli inequalities are quasiconformal mappings \cite{A66,V71} and the conformal modules method is one of the basic methods in the geometric theory of quasiconformal mappings \cite{V71}.

The topological mappings defined by $p$-capacity inequalities were firstly studied in \cite{Ger69}. In this work \cite{Ger69} the notion of $p$-distortion of mappings  was introduced and the Lipschitz property of such mappings was proved in the case $n-1<p<\infty$. The topological mappings defined by $(p,q)$-capacity inequalities were studied in \cite{U93} in connections with composition operators on Sobolev spaces. In \cite{U93,VU98} the weak differentiability of inverse mappings and its H\"older continuity was proved in the case $n-1<q<p<\infty$. Continuous mappings satisfying $(p,q)$-capacity inequalities were considered in \cite{UV10}. In this work \cite{UV10}  Liouville type theorems were proved and removability properties of singular sets were considered.  In the recent paper \cite{V21} (see also \cite{UV08,V20}) mappings which satisfy weighted $(p,q)$-capacity inequalities were considered in connection with problems of the geometric theory of composition operators on Sobolev spaces \cite{U93,VU02,VU04}.

Significant contributions to the theory of mappings defined by moduli inequalities belong to the Donetsk geometric mapping theory school. In particular, for mappings satisfying to weighted Poletsky's type inequalities for $p$-modulus of families of curves, their differentiability almost everywhere and the local integrability of partial derivatives were established \cite{SS11}. Some problems concerning the local behavior of such mappings and the problem of removability of an isolated singular
point were investigated in  \cite{GSS15,GSS16,Sev11}. In particular, mappings with a distortion of the
modulus of order $n-1<p<n $ have no essential singular points, that fundamentally distinguishes them from analytic functions and mappings with bounded distortion, see~\cite[Theorem~1.1]{GSS15}.
Let us note that some results concerning of weighted inequalities with respect to the $p$-modulus for such classes of mappings were obtained in~\cite{SS14}. It should be noted that the similar theory of mappings was developing independently in the context of directly weighted $p$-moduli, that can be found in works \cite{Cr16,Cr21}.

Composition operators on Sobolev spaces arise in the geometric analysis of Sobolev spaces \cite{GGu,GS82,GU} and are closely connected with the quasiconformal Reshetnyak problem \cite{VG75}. In series of works \cite{U93,V88,VU98,VU02,VU04} and \cite{GGR95} was founded the geometric theory of composition operators on Sobolev spaces. This theory has significant applications in the spectral theory of elliptic equations, see for example, \cite{GPU18_3,GU16,GU17}. 

The composition operators on Sobolev spaces are generated by weak $(p,q)$-quasiconformal mappings \cite{GGR95,U93,VU98} and allow characterization in the terms of capacity inequalities. Hence there is a connection between the geometric theory of composition operators on Sobolev spaces and the theory of mappings defined moduli inequalities. In \cite{MU21,MU21_2} were considered $Q$-homeomorphisms \cite{MRSY09} in connection with composition operators on Sobolev spaces. 

In the present work we give connection between mappings defined by weighted moduli inequalities and weak $(p,q)$-quasiconformal mappings (mappings generate composition operators on Sobolev spaces). Namely we prove the following statement:

Let a homeomorphism $\varphi: \Omega \to \widetilde{\Omega}$ satisfies to the moduli inequality
\begin{equation*}
M_p\left(\varphi \Gamma\right)\leqslant \int\limits_{\Omega} Q(x)\cdot
\rho^{p}(x)dx, \,\,n-1<p<\infty,
\end{equation*}
with a non-negative function $Q\in L_1(\Omega)$. Then $\varphi$ generates the bounded composition operator
$$
\varphi^{\ast}: L^1_{p'}(\widetilde\Omega) \to L^1_{1}(\Omega), \,\,p'=p/(p-n+1).
$$

The inverse assertion states:

Let a homeomorphism $\varphi: \Omega \to \widetilde{\Omega}$  generates the bounded composition operator
$$
\varphi^{\ast}: L^1_p(\widetilde\Omega) \to L^1_{n-1}(\Omega), \qquad
n-1< p< \infty\,.
$$
Suppose also that the mapping $\varphi$ satisfies Luzin's $N$-property. Then
\begin{equation*}
M_{p'}\left(\varphi \Gamma\right)\leqslant \int\limits_{\Omega} Q(x)\cdot
\rho^{p'}(x)dx, \,\,p^{\,\prime}=\frac{p}{p-n+1},
\end{equation*}
with a non-negative function $Q\in L_1(\Omega)$.

The suggested methods are based on the geometric theory of composition operators on Sobolev spaces and moduli inequalities.

\section{Sobolev spaces and composition operators}

\subsection{Sobolev spaces}
Let $\Omega$ be an open subset of $\mathbb R^n$. The Sobolev space
$W^1_p(\Omega)$, $1\leqslant p\leqslant\infty$, is defined \cite{M}
as a Banach space of locally integrable weakly differentiable
functions $f:\Omega\to\mathbb{R}$ equipped with the following norm:
\[
\|f\mid W^1_p(\Omega)\|=\| f\mid L_p(\Omega)\|+\|\nabla f\mid L_p(\Omega)\|,
\]
where $\nabla f$ is the weak gradient of the function $f$, i.~e. $ \nabla f = (\frac{\partial f}{\partial x_1},...,\frac{\partial f}{\partial x_n})$.

The homogeneous seminormed Sobolev space $L^1_p(\Omega)$,
$1\leqslant p\leqslant\infty$, is defined as a space of locally
integrable weakly differentiable functions $f:\Omega\to\mathbb{R}$
equipped with the following seminorm:
\[
\|f\mid L^1_p(\Omega)\|=\|\nabla f\mid L_p(\Omega)\|.
\]

In the Sobolev spaces theory, a crucial role belongs to capacities and associated with capacities an outer (capacitary) measure. This measure  that regulate some natural properties (for example convergence properties) of corresponding Sobolev spaces \cite{HKM,M}. In accordance to this approach, elements of Sobolev spaces $W^1_p(\Omega)$ are equivalence classes up to a set of $p$-capacity zero \cite{MH72}. 

Recall the definition of the capacity \cite{GResh,HKM,M}.
Suppose $\Omega$ is an open set in $\mathbb R^n$ and  $F\subset\Omega$ is a compact set. The $p$-capacity of $F$ with respect to $\Omega$ is defined by
\begin{equation*}
\cp_p(F;\Omega) =\inf\{\|\nabla f|L_p(\Omega)\|^p\},
\end{equation*} 
where the infimum is taken over all functions $f\in C_0(\Omega)\cap L^1_p(\Omega)$ such that $f\geq 1$ on $F$ and which are called admissible functions for the compact set $F\subset\Omega$. 
If 
$\subset\Omega$ 
is an open set, we define
\begin{equation*}
\cp_{p}(U;\Omega)=\sup\{\cp_{p}
(F;\Omega)\,:\,F\subset U,\,\, F\,\,\text{is compact}\}.
\end{equation*}

In the case of an arbitrary set 
$E\subset\Omega$
we define the inner $p$-capacity 
\begin{equation*}
\underline{\cp}_{p}(E;\Omega)=\sup\{\cp_{p}(F;\Omega)\, :\,\,F\subset E\subset\Omega,\,\, F\,\,\text{is compact}\},
\end{equation*}
and the outer $p$-capacity 
\begin{equation*}
\overline{\cp}_{p}(E;\Omega)=\inf\{\cp_{p}(U;\Omega)\, :\,\,E\subset U\subset\Omega,\,\, U\,\,\text{is open}\}.
\end{equation*}

A set $E\subset\Omega$ is called $p$-capacity measurable, if $\underline{\cp}_p(E;\Omega)=\overline{\cp}_p(E;\Omega)$. Let $E\subset\Omega$ be a $p$-capacity measurable set. The value
$$
\cp_p(E;\Omega)=\underline{\cp}_p(E;\Omega)=\overline{\cp}_p(E;\Omega)
$$
is called the $p$-capacity measure of the set $E\subset\Omega$.

The mapping $\varphi:\Omega\to\mathbb{R}^{n}$ belongs to the Sobolev space $W^1_{1,\loc}(\Omega,\mathbb R^n)$, if its coordinate functions belongs to $W^1_{1,\loc}(\Omega)$. In this case, the formal Jacobi matrix $D\varphi(x)$ and its determinant (Jacobian) $J(x,\varphi)$
are well defined at almost all points $x\in\Omega$. The norm $|D\varphi(x)|$ is the operator norm of $D\varphi(x)$,

Let us recall the change of variable formula in the Lebesgue integral \cite{F69, H93}.
Suppose a homeomorphism $\varphi : \Omega\to \mathbb R^n$ is such that
there exists a collection of closed sets $A_k\subset A_{k+1}\subset \Omega$, $k=1,2,...$, for which restrictions $\varphi \vert_{A_k}$ are Lipschitz mappings on the sets $A_k$ and
$$
\biggl|\Omega\setminus\sum\limits_{k=1}^{\infty}A_k\biggr|=0.
$$
Then there exists a measurable set $S\subset \Omega$, $|S|=0$, such that  the homeomorphism $\varphi:\Omega\setminus S \to \mathbb R^n$ has the Luzin $N$-property (the image of a set of measure zero has measure zero) and the change of variable formula
\begin{equation}
\label{chvf}
\int\limits_E f\circ\varphi (x) |J(x,\varphi)|~dx=\int\limits_{\mathbb R^n\setminus \varphi(S)} f(y)~dy
\end{equation}
holds for every measurable set $E\subset \Omega$ and every non-negative measurable function $f: \mathbb R^n\to\mathbb R$.

Note, that Sobolev homeomorphisms of the class $W^1_{1,\loc}(\Omega)$ satisfy to the conditions of the change of variable formula \cite{H93} and, therefore, for Sobolev homeomorphisms the change of variable formula \eqref{chvf} holds.

If the mapping $\varphi$ possesses the Luzin $N$-property, then
$|\varphi (S)|=0$ and the second integral can be rewritten as the
integral on $\mathbb R^n$. Note, that Sobolev homeomorphisms of the
class $L^1_p(\Omega)$, $p\geqslant n$, possess the Luzin
$N$-property \cite{VGR}.

\subsection{Composition operators and regularity of inverse mappings}

Let $\Omega$ and $\widetilde{\Omega}$ be domains in the Euclidean space $\mathbb R^n$. Then a homeomorphism $\varphi:\Omega\to\widetilde{\Omega}$ generates a bounded composition
operator
\[
\varphi^{\ast}:L^1_p(\widetilde{\Omega})\to
L^1_q(\Omega),\,\,\,1\leqslant q\leqslant p\leqslant\infty,
\]
by the composition rule $\varphi^{\ast}(f)=f\circ\varphi$, if for
any function $f\in L^1_p(\widetilde{\Omega})$, the composition $\varphi^{\ast}(f)\in L^1_q(\Omega)$
defined quasi-everywhere in $\Omega$ and there exists a constant $K_{p,q}(\varphi;\Omega)<\infty$ such that
\[
\|\varphi^{\ast}(f)\mid L^1_q(\Omega)\|\leqslant
K_{p,q}(\varphi;\Omega)\|f\mid L^1_p(\widetilde{\Omega})\|.
\]

Recall that the $p$-dilatation \cite{Ger69} of a Sobolev mapping $\varphi: \Omega\to \widetilde{\Omega}$ at a point $x\in\Omega$ defined as
$$
K_p(x)=\inf \{k(x): |D\varphi(x)|\leqslant k(x)
|J(x,\varphi)|^{\frac{1}{p}}\}.
$$

The following theorem gives a characterization of composition operators in terms of integral characteristics of mappings of finite distortion. Recall that a weakly differentiable mapping $\varphi:\Omega\to\mathbb{R}^{n}$ is the mapping of finite distortion if $D\varphi(x)=0$ for almost all $x$ from $Z=\{x\in\Omega: J(x,\varphi)=0\}$ \cite{VGR}.

\begin{thm}
\label{CompTh} A homeomorphism $\varphi:\Omega\to\widetilde{\Omega}$
between two domains $\Omega$ and $\widetilde{\Omega}$ generates a bounded composition
operator
\[
\varphi^{\ast}:L^1_p(\widetilde{\Omega})\to
L^1_{q}(\Omega),\,\,\,1\leqslant q\leqslant p\leqslant\infty,
\]
 if and only if $\varphi\in W^1_{q,\loc}(\Omega)$, has finite distortion,
and
\[
K_{p,q}(\varphi;\Omega) := \|K_p \mid L_{\kappa}(\Omega)\|<\infty, \,\,1/q-1/p=1/{\kappa}\,\,(\kappa=\infty, if p=q).
\]
The norm of the operator $\varphi^\ast$ is estimated as
$\|\varphi^\ast\| \leqslant K_{p,q}(\varphi;\Omega)$.
\end{thm}

This theorem in the case $p=q=n$ was proved in \cite{VG75}, the case $p=q>n$ was proved in \cite{V88} (see, also \cite{GGR95}). The general case $q\leq p<\infty$ was proved in \cite{U93} (see, also \cite{VU98}) and the limit case $p=\infty$ was considered in \cite{GU10}.

{\it Let us recall, that homeomorphisms $\varphi:\Omega\to\widetilde{\Omega}$ which satisfy conditions of Theorem~\ref{CompTh}  are called as weak $(p,q)$-quasiconformal mappings \cite{GGR95,VU98}. }

In the case of weak $(p,q)$-quasiconformal mappings, the following composition duality property was introduced in \cite{U93} (the detailed proof can be found in  \cite{GU19}):

\begin{thm}
\label{CompThD} Let a homeomorphism $\varphi:\Omega\to\widetilde{\Omega}$
between two domains $\Omega$ and $\widetilde{\Omega}$ generates a bounded composition
operator
\[
\varphi^{\ast}:L^1_p(\widetilde{\Omega})\to
L^1_{q}(\Omega),\,\,\,n-1<q \leqslant p< \infty,
\]
then the inverse mapping $\varphi^{-1}:\widetilde{\Omega}\to\Omega$ generates a bounded composition operator
\[
\left(\varphi^{-1}\right)^{\ast}:L^1_{q'}(\Omega)\to L^1_{p'}(\widetilde{\Omega}),
\]
where $p'=p/(p-n+1)$, $q'=q/(q-n+1)$. In the case $n=2$ this theorem is correct for $1\leqslant q\leqslant p<\infty$, the inverse assertion is also correct and $p''=p$.
\end{thm}

The limit case of this theorem $p=\infty$ and $q=n-1$ was considered in \cite{GU10} in the frameworks of the weak inverse theorem for Sobolev spaces.

In the present work we consider the case $q=n-1<p<\infty$.

\begin{thm}
\label{CompThDLim} Let a homeomorphism $\varphi:\Omega\to\widetilde{\Omega}$
between two domains $\Omega$ and $\widetilde{\Omega}$ possesses the Luzin $N$-property and generates a bounded composition
operator
\[
\varphi^{\ast}:L^1_p(\widetilde{\Omega})\to
L^1_{n-1}(\Omega),\,\,\,n-1<p< \infty,
\]
then the inverse mapping $\varphi^{-1}:\widetilde{\Omega}\to\Omega$ generates a bounded composition operator
\[
\left(\varphi^{-1}\right)^{\ast}:L^1_{\infty}(\Omega)\to L^1_{p'}(\widetilde{\Omega}),
\]
where $p'=p/(p-n+1)$. In the case $n=2$ the inverse assertion is also correct and $p''=p$.
\end{thm} 

\begin{proof}
Since $\varphi:\Omega\to\widetilde{\Omega}$ generates a bounded composition
operator
\[
\varphi^{\ast}:L^1_p(\widetilde{\Omega})\to
L^1_{n-1}(\Omega),\,\,\,n-1<p< \infty,
\]
then by \cite{U93} the mapping $\varphi\in W^1_{n-1,\loc}(\Omega)$ and has a finite distortion. Because $\varphi$ possesses the Luzin $N$-property, then the inverse mapping belongs to $W^1_{1,\loc}(\widetilde{\Omega})$ and is a mapping of finite distortion \cite{GU10}.
Hence 
\[
|D\varphi^{-1}(y)|\leq\frac{|D\varphi(x)|^{n-1}}{|J(x,\varphi)|},
\]
for almost all $x\in \Omega\setminus \left(S\cup Z\right)$, $y=\varphi(x)\in \widetilde{\Omega}\setminus \varphi\left(S\cup Z\right)$, and  $|D\varphi^{-1}(y)|=0$ for almost all $y\in \varphi(S)$, where $Z=\{x\in\Omega: J(x,\varphi)=0\}$ and $S$ is the singular set in the change of variables formula (\ref{chvf}).
Because a measure of $S$ is zero and the mapping $\varphi$ has the Luzin $N$-property, then a measure of $\varphi(S)$ is also zero.

Therefore
\begin{multline*}
\int\limits_{\widetilde{\Omega}}|D\varphi^{-1}(y)|^{p'}~dy=\int\limits_{\widetilde{\Omega}\setminus \varphi\left(S\cup Z\right)}|D\varphi^{-1}(y)|^{p'}~dy\\
\leq \int\limits_{\widetilde{\Omega}\setminus \varphi\left(S\cup Z\right)}\left(\frac{|D\varphi(\varphi^{-1}(y))|^{n-1}}{|J(\varphi^{-1}(y),\varphi)|}\right)^{p'}~dy
=\int\limits_{\Omega\setminus \left(S\cup Z\right)}\left(\frac{|D\varphi(x)|^{n-1}}{|J(x,\varphi)|}\right)^{p'}|J(x,\varphi)|~dx\\
=
\int\limits_{\Omega}\left(\frac{|D\varphi(x)|^{p}}{|J(x,\varphi)|}\right)^{\frac{n-1}{p-(n-1)}}~dx<\infty,
\end{multline*}
by Theorem~\ref{CompTh}. 

Hence \cite{GU10} $\varphi^{-1}:\widetilde{\Omega}\to\Omega$ generates a bounded composition operator
\[
\left(\varphi^{-1}\right)^{\ast}:L^1_{\infty}(\Omega)\to L^1_{p'}(\widetilde{\Omega}),
\]
where $p'=p/(p-n+1)$.

\end{proof}

\section{Composition operators and moduli inequalities}

\subsection{Modulus and capacity}
Let $\Gamma$ be a family of curves in $\mathbb R^n$. Denote by $adm(\Gamma)$ the set of Borel functions (admissible functions)
$\rho: \mathbb R^n\to[0,\infty]$ such that the inequality
$$
\int\limits_{\gamma}\rho~ds\geqslant 1
$$
holds for locally rectifiable curves $\gamma\in\Gamma$.

Let $\Gamma$ be a family of curves in $\overline{\mathbb R^n}$, where $\overline{\mathbb R^n}$ is a one point compactification of the Euclidean space $\mathbb R^n$. The quantity
$$
M_p(\Gamma)=\inf\int\limits_{\mathbb R^n}\rho^{p}~dx
$$
is called the $p$-module of the family of curves $\Gamma$ \cite{MRSY09}. The infimum is taken over all admissible functions
$\rho\in adm(\Gamma)$.

Let $\Omega$ be a bounded domain in $\mathbb R^n$ and $F_0, F_1$ be disjoint non-empty compact sets in the
closure of $\Omega$. Let $M_p(\Gamma(F_0,F_1;\Omega))$ stand for the
moduli of a family of curves which connect $F_0$ and $F_1$ in $\Omega$. Then \cite{MRSY09}
\begin{equation}\label{eq2}
M_p(\Gamma(F_0,F_1;\Omega)) = \cp_{p}(F_0,F_1;\Omega)\,,
\end{equation}
where $\cp_{p}(F_0,F_1;\Omega)$ is a $p$-capacity of the condenser $(F_0,F_1;\Omega)$ \cite{M}.

Suppose that a homeomorphism  $\varphi: \Omega\to\widetilde{\Omega}$ between two domains $\Omega$ and $\widetilde{\Omega}$ satisfy to the moduli inequality
\begin{equation}
\label{ME}
M_p\left(\varphi \Gamma\right)\leqslant \int\limits_{\Omega} Q(x)\cdot
\rho^{p}(x)dx
\end{equation}
with a non-negative measurable function $Q$ for every family $\Gamma$ of rectifiable curves in $\Omega$ and every admissible function $\rho$ for $\Gamma$. Such homeomorphisms  called  $Q$-homeomorphisms. 

The next section is devoted to connections between $Q$-homeomorphisms and the composition operators in the case 
$Q \in L_1(\Omega)$.

\subsection{Composition operators and $Q$-homeomorphisms}

Firstly we define two dilatation functions for Sobolev mappings of finite distortion $\varphi: \Omega \to \widetilde\Omega$.

\noindent
The outer  $p$-dilatation is the following quantity
$$
K^O_p(x,\varphi)= 
\begin{cases}
\frac{|D\varphi(x)|^p}{|J(x,\varphi)|},& \,\, J(x,\varphi)\ne 0,\\
0,& \,\, J(x,\varphi)= 0.
\end{cases}
$$
The inner $p$-dilatation is the following quantity
$$
K^I_p(x,\varphi)=
\begin{cases}
\frac{|J(x,\varphi)|}{l(D\varphi(x))^p},& \,\, J(x,\varphi)\ne 0,\\
0,& \,\, J(x,\varphi)= 0,
\end{cases}
$$
where $l(D\varphi(x))=\min\limits_{|h|=1}|D\varphi(x)\cdot h|$ for almost all $x\in\Omega$.

\begin{thm}\label{mod-comp}
Let a homeomorphism $\varphi: \Omega \to \widetilde{\Omega}$ satisfies to the moduli inequality
\begin{equation}
\label{ME1}
M_p\left(\varphi \Gamma\right)\leqslant \int\limits_{\Omega} Q(x)\cdot
\rho^{p}(x)dx, \,\,n-1<p<\infty,
\end{equation}
with a non-negative function $Q\in L_1(\Omega)$. Then $\varphi$ generates the bounded composition operator
$$
\varphi^{\ast}: L^1_{p'}(\widetilde\Omega) \to L^1_{1}(\Omega), \,\,p'=p/(p-n+1).
$$
\end{thm}

\begin{proof}
Because $\varphi$ satisfies the moduli inequality with $Q \in L_1(\Omega)$, then by \cite{SS11} the mapping
$\varphi\in\ACL(\Omega)$, has finite distortion and the inequality
\begin{equation}
\label{Q-dist} |D\varphi(x)|^p \leqslant C(n,p)
|J(x,\varphi)|^{p-n+1}Q^{n-1}(x)
\end{equation}
holds for almost all $x\in\Omega$. Hence
$$
\left(\frac{|D\varphi(x)|^{\frac{p}{p-n+1}}}{|J(x,\varphi)|}\right)^{\frac{p-n+1}{n-1}}\leqslant
Q(x)\,\,\text{for almost all}\,\,x\in\Omega.
$$

Since $Q \in L_1(\Omega)$, we have
$$
\int\limits_{\Omega}\left(\frac{|D\varphi(x)|^{\frac{p}{p-n+1}}}{|J(x,\varphi)|}\right)^{\frac{p-n+1}{n-1}}~dx\leqslant
\int\limits_{\Omega}Q(x)~dx<\infty.
$$
Then, by Theorem~\ref{CompTh} the mapping $\varphi$ generates the bounded composition operator $\varphi^{\ast}: L^1_{p'}(\widetilde\Omega) \to L^1_{q'}(\Omega)$, where $p'=p/(p-n+1)$ and the number $q'$ is defined by
$$
\frac{q'}{p'-q'}=\frac{p-n+1}{n-1}\,\, \Rightarrow  \,\,q'=1.
$$
\end{proof}

In \cite{GSS17} the integrability of Jacobains of open discrete mappings with controlled p-module was considered. We note, that as the consequence of Theorem~\ref{mod-comp} we obtain that homeomorphisms which satisfy the weighted $p$-moduli inequality (\ref{ME1}) posses  measure distortion properties as weak $(p',1)$ quasiconformal mappings \cite{VU98,VU02}. 

Now using the composition duality property in the case of planar domains $\Omega,\widetilde{\Omega}\subset\mathbb R^2$ we obtain:

\begin{thm}
Let a homeomorphism $\varphi: \Omega \to \widetilde{\Omega}$, $\Omega,\widetilde{\Omega}\subset\mathbb R^2$, satisfies the moduli inequality
\begin{equation*}
M_p\left(\varphi \Gamma\right)\leqslant \int\limits_{\Omega}
Q(x)\cdot \rho^{p}(x)dx, \,\,1\leqslant p<\infty,
\end{equation*}
with a non-negative function $Q\in L_1(\Omega)$. Then the inverse mapping $\varphi^{-1}:\widetilde{\Omega}\to \Omega$ generates the bounded composition operator
$$
\left(\varphi^{-1}\right)^{\ast}: L^1_{\infty}(\Omega) \to L^1_{p}(\widetilde\Omega).
$$
In particular, $\varphi^{-1}\in L^1_{p}(\widetilde{\Omega})$.
\end{thm}

Now we prove the inverse property.

\begin{thm}
Let a homeomorphism $\varphi: \Omega \to \widetilde{\Omega}$  generates the bounded composition operator
$$
\varphi^{\ast}: L^1_p(\widetilde\Omega) \to L^1_{n-1}(\Omega), \qquad
n-1< p< \infty\,.
$$
Suppose that the mapping $\varphi$ satisfies Luzin's $N$-property. Then
$\varphi$ is a $Q$-homeomorphism with respect to
$p^{\,\prime}$-modulus with $Q(x)=K^I_{p^{\,\prime}}(x,\varphi)\in L_1(\Omega),$
where $p^{\,\prime}=\frac{p}{p-n+1}.$
\end{thm}

\begin{proof}
On the first step we prove that under conditions of the theorem the inner distortion function $K_{p^{\,\prime}}^I(x, \varphi)\in L_1(\Omega)$.
Since $\varphi$ generates the bounded composition operator
$$
\varphi^{\ast}: L^1_p(\widetilde\Omega) \to L^1_{n-1}(\Omega), \qquad
n-1< p< \infty\,.
$$
then by Theorem~\ref{CompThDLim} the inverse mapping $\varphi^{-1}:\widetilde\Omega\to\Omega$ generates a bounded composition operator
$$
\left(\varphi^{-1}\right)^{\ast} : L^1_{\infty}(\Omega)\to L^1_{p'}(\widetilde\Omega), \,\,p'=\frac{p}{p-n+1}
$$
and belongs to the Sobolev space $L^1_{p'}(\widetilde\Omega)$ \cite{GU10}.

Hence because the inverse mapping is a mapping of finite distortion, then by \cite{GU21}
\begin{multline*}
\int\limits_{\Omega}\frac{|J(x,\varphi)|}{l(D\varphi(x))^{p'}}~dx=
\int\limits_{\Omega\setminus Z}\frac{|J(x,\varphi)|}{l(D\varphi(x))^{p'}}~dx
=\int\limits_{\Omega\setminus Z} |D\varphi^{-1}(\varphi(x))|^{p'}|J(x,\varphi)|~dx\\
=\int\limits_{\widetilde{\Omega}}|D\varphi^{-1}(y)|^{p'}~dy<\infty.
\end{multline*}

Now by the definition of the $Q$-homeomorphism with respect to
$p^{\,\prime}$-modulus, we have to show, that for every family
$\Gamma$ of curves in $\Omega$ and every $\rho \in adm(\Gamma)$
$$
M_{p^{\,\prime}}(\varphi\Gamma) \leqslant \int\limits_\Omega
K^I_{p^{\,\prime}}(x,\varphi) \rho^{\,p^{\,\prime}}(x) \, dx.
$$

First, note, that by theorem \ref{CompTh}, $\varphi\in
W^1_{n-1,\loc}(\Omega)$. Also, $\varphi^{-1} \in
W^1_{p^{\,\prime},\loc}(\widetilde\Omega)$ by theorem \ref{CompThD}.
It implies, that $\varphi^{-1} \in
\ACL^{p^{\,\prime}}_{\loc}(\widetilde\Omega)$, is differentiable
a.e. (see \cite[Lemma~3]{Va65}).

By Fuglede's theorem (\cite{V71}, p.~95), if $\widetilde\Gamma$ is
the family of all curves $\gamma \in \varphi\Gamma$ for which
$\varphi^{-1}$ is absolutely continuous on all closed subcurves of
$\gamma$, then $M_{p^{\,\prime}}(\varphi\Gamma) =
M_{p^{\,\prime}}(\widetilde\Gamma)$. Then, for given $\rho \in
adm(\Gamma)$, one consider
$$
\widetilde\rho(y) =
\begin{cases}
\rho(\varphi^{-1}(y))|D\varphi^{-1}(y)|,& \,\, y \in \widetilde\Omega,\\
0,& \,\, \text{otherwise}.
\end{cases}
$$
Then, for $\widetilde\gamma \in \widetilde\Gamma$
$$
\int\limits_{\widetilde\gamma}\widetilde\rho \, ds \geqslant
\int\limits_{\varphi^{-1}\circ\widetilde\gamma} \rho \, ds \geqslant
1,
$$
and consequently $\widetilde\rho \in adm(\widetilde\Gamma)$.

We denote by $Z_0$ the set of all points $y\in\widetilde{\Omega},$
where $J_{\varphi^{-1}}(y)=0.$ By change of variable formula (see
\cite[Theorem~3.2.5]{F69}), we obtain that
\begin{multline*}
M_{p^{\,\prime}}(\varphi\Gamma) = M_{p^{\,\prime}}(\widetilde\Gamma)
\leqslant
\int\limits_{\widetilde\Omega}\widetilde{\rho}^{p^{\,\prime}} \, dy \\
=
\int\limits_{\widetilde\Omega}\rho^{p^{\,\prime}}(\varphi^{-1}(y))|D\varphi^{-1}(y)|^{p^{\,\prime}}
\, dy = \int\limits_{\widetilde\Omega\setminus
Z_0}\frac{\rho^{p^{\,\prime}}(\varphi^{-1}(y))}
{l(D\varphi(\varphi^{-1}(y)))^{p^{\,\prime}}} \, dy \\
= \int\limits_{\widetilde\Omega}
\rho^{p^{\,\prime}}(\varphi^{-1}(y))
K^I_{p^{\,\prime}}(\varphi^{-1}(y), \varphi) J_{\varphi^{-1}}(y) \,
dy \leqslant \int\limits_\Omega K_{p^{\,\prime}}^I(x, \varphi)
\rho^{p^{\,\prime}}(x) \, dx
\end{multline*}
which completes the proof.
\end{proof}

In the planar case $\Omega,\widetilde{\Omega}\subset\mathbb R^2$ we have the following theorem:

\begin{thm}\label{mod-comp-plane}
Let $\varphi: \Omega \to \widetilde{\Omega}$ be a homeomorphism  of domains $\Omega,\widetilde{\Omega}\subset\mathbb R^2$. Then $\varphi$ satisfies the moduli inequality
\begin{equation}
\label{ME2}
M_p\left(\varphi \Gamma\right)\leqslant \int\limits_{\Omega} Q(x)\cdot
\rho^{p}(x)dx, \,\,1<p<\infty,
\end{equation}
with a non-negative function $Q\in L_1(\Omega)$, if and only if $\varphi$ generates the bounded composition operator
$$
\varphi^{\ast}: L^1_{p'}(\widetilde\Omega) \to L^1_{1}(\Omega), \,\,p'=p/(p-1).
$$
\end{thm}

\vskip 0.3cm

Vladimir Gol'dshtein; Department of Mathematics, Ben-Gurion University of the Negev, P.O.Box 653, Beer Sheva, 8410501, Israel 
 
\emph{E-mail address:} \email{vladimir@math.bgu.ac.il} \\     

Evgeny Sevost'yanov; Department of Mathematical Analysis, Zhytomyr Ivan Franko State University, 40 Bol'shaya Berdichevskaya Str., Zhytomyr, 10008, Ukraine      
 
\emph{E-mail address:} \email{esevostyanov2009@gmail.com} \\    

Alexander Ukhlov; Department of Mathematics, Ben-Gurion University of the Negev, P.O.Box 653, Beer Sheva, 8410501, Israel 
							
\emph{E-mail address:} \email{ukhlov@math.bgu.ac.il

\end{document}